\documentclass{amsproc}

\usepackage{amsmath}
\usepackage{amssymb}
\usepackage[mathcal]{eucal}
\usepackage[bb=pazo]{mathalfa}
\usepackage{rotating}
\usepackage{graphicx}
\usepackage[tableposition=below]{caption}
\usepackage{pigpen}

\DeclareGraphicsExtensions{.png,.pdf,.eps}
\usepackage[all,2cell,cmtip]{xy}
\usepackage{tikz}
\usepackage{color}
\usepackage{longtable}
\usepackage{soul}
\newtheorem{theorem}{Theorem}[section]
\newtheorem{lemma}[theorem]{Lemma}
\newtheorem{corollary}[theorem]{Corollary}
\newtheorem{proposition}[theorem]{Proposition}

\newtheorem{notation}[theorem]{Notation}

\theoremstyle{definition}

\newtheorem{definition}[theorem]{Definition}

\newtheorem{remark}[theorem]{Remark}

\newtheorem{setup}[theorem]{Setup}
\newtheorem{example}[theorem]{Example}

\newtheorem{construction}{Construction}

\def\C{{\mathbb C}}

\def\P{{\mathbb P}}

\def\R{{\mathbb R}}
\def\Z{{\mathbb Z}}

\def\cD{{\mathcal D}}
\def\cE{{\mathcal E}}

\def\cM{{\mathcal M}}

\def\cO{{\mathcal{O}}}
\def\cP{{\mathcal{P}}}

\def\cU{{\mathcal U}}

\def\fg{{\mathfrak g}}
\def\fh{{\mathfrak h}}
\def\fp{{\mathfrak p}}
\def\fb{{\mathfrak b}}

\def\GL{\operatorname{Gl}}
\def\PGL{\operatorname{PGl}}
\def\SL{\operatorname{Sl}}
\def\SO{\operatorname{SO}}
\def\SP{\operatorname{Sp}}

\def\rat{\dashrightarrow}

\def\operatorname#1{\mathop{\rm #1}\nolimits}

\def\Aut{\operatorname{Aut}}

\def\Hom{\operatorname{Hom}}

\def\Pic{\operatorname{Pic}}

\def\rk{\operatorname{rk}}

\def\rat{\operatorname{RatCurves}}

\def\NE{{\operatorname{NE}}}

\def\Nef{{\operatorname{Nef}}}

\newcommand{\cNE}[1]{\overline{\NE}(#1)}

\newcommand{\sgn}{\operatorname{sgn}}
\newcommand{\Chi}{\ensuremath \raisebox{2pt}{$\chi$}}

\newcommand{\pb}{\ar@{}[dr]|(.25){\text{\pigpenfont J}}}

\makeatletter
 \newcommand*\tl[1]{\mathpalette\wthelper{#1}}
\newcommand*\wthelper[2]{%
        \hbox{\dimen@\accentfontxheight#1%
                \accentfontxheight#11.15\dimen@
                $\m@th#1\widetilde{#2}$%
                \accentfontxheight#1\dimen@
        }%
}

\newcommand*\accentfontxheight[1]{%
        \fontdimen5\ifx#1\displaystyle
                \textfont
        \else\ifx#1\textstyle
                \textfont
        \else\ifx#1\scriptstyle
                \scriptfont
        \else
                \scriptscriptfont
        \fi\fi\fi3
}
\makeatother

\DeclareMathOperator{\ch}{\mathrm{Chains}}

\makeindex

\begin{document}

\title[Flag bundles on Fano manifolds]{Flag bundles on Fano manifolds }

\author[Occhetta]{Gianluca Occhetta}
\address{Dipartimento di Matematica, Universit\`a di Trento, via
Sommarive 14 I-38123 Povo di Trento (TN), Italy} \thanks{First author supported by PRIN project ``Geometria delle variet\`a algebriche''.
First and second author supported by the Department of Mathematics of the University of Trento. Second and third author supported by the Polish National Science Center project 2013/08/A/ST1/00804.}
\email{gianluca.occhetta@unitn.it, eduardo.solaconde@unitn.it}

\author[Sol\'a Conde]{Luis E. Sol\'a Conde}
\subjclass[2010]{Primary 14J45; Secondary 14E30, 14M15, 14M17}
%
%
\author[Wi\'sniewski]{Jaros\l{}aw A. Wi\'sniewski}
\address{Instytut Matematyki UW, Banacha 2, 02-097 Warszawa, Poland}
\email{J.Wisniewski@uw.edu.pl}

\begin{abstract}
As an application of a recent characterization 
of complete flag manifolds as Fano manifolds having only $\P^1$-bundles as elementary contractions, we consider here the case of a Fano manifold $X$ of Picard number one supporting an unsplit family of rational curves whose subfamilies parametrizing curves through a fixed point are rational homogeneous, and we prove that $X$ is homogeneous. In order to do this, we first study minimal sections on flag bundles over the projective line, and discuss how Grothendieck's theorem on principal bundles allows us to describe a flag bundle upon some special sections.
\end{abstract}

\maketitle

\section{Introduction}\label{sec:intro}

Although vector bundles and their projectivizations play an important role in every branch of Algebraic Geometry, a framework in which these objects are especially manageable is the one of algebraic varieties containing rational curves. The main reason for this is the fact, attributed to C. Segre for rank two (see \cite[p. 44]{OSS} for historical remarks on the general statement), that every vector bundle over the projective line $\P^1$ is isomorphic to a direct sum of line bundles. However, it was not until Alexandre Grothendieck's celebrated paper \cite{Gro1} that this theorem achieved an optimal form, in the framework of principal bundles. In fact, Grothendieck shows (by reducing the general case, via the adjoint representation of $G$, to the study of orthogonal bundles over $\P^1$) that every principal $G$-bundle --for $G$ reductive-- is diagonalizable (see Theorem \ref{thm:groth} below).

Our interest in this topic comes from its relation with certain homogeneity criteria that we have been recently considering  (\cite{MOSW,OSWW}) within a project whose goal is the Campana--Peternell conjecture, which predicts that every Fano manifold with nef tangent bundle is rational homogeneous. In a nutshell, we showed that flag manifolds are characterized within the class of Fano manifolds by having only $\P^1$-bundles as elementary contractions. In particular, one may then try to use this result to prove the homogeneity of a certain Fano manifold $X$ by ``untangling'' its families of extremal rational curves, constructing upon $X$ another Fano manifold $\tl{X}$ dominating it and satisfying the above property. 

More concretely, this ``bottom-up'' strategy may be roughly described as follows: we start from a Fano manifold whose homogeneity we want to check; we consider a (not necessarily complete) proper dominating family of minimal rational curves in $X$, $\cM\stackrel{p}{\longleftarrow}\cU\stackrel{q}{\longrightarrow} X$, and ask ourselves whether $\tl{X}=\cU$ is again a Fano manifold and, in this case, proceed by substituting $X$ by $\tl{X}$. If this procedure can be carried out until we get to a Fano manifold  in which all the families of minimal rational curves are $\P^1$-bundles, then the original variety $X$ will be homogeneous.

This process can be shortened in the particular case in which $q$ is smooth and its fibers $\cM_x:=q^{-1}(x)$  are homogeneous, since rational homogeneous bundles are determined by principal bundles, leading us immediately to a complete flag bundle $\tl{\cU}$ dominating $\cU$ (see Section \ref{sec:notn} below). In order to check that $\tl{\cU}$ is in fact a complete flag, we need to study sections of the bundle $\tl{\cU}$ over minimal rational  curves in $X$, which is our motivation to give a ``rational curves oriented'' interpretation of Grothendieck's theorem. 

Section \ref{sec:groth}  is devoted to this topic. More concretely, we will study $G/B$-bundles (with $G$ semisimple and $B\subset G$ a Borel subgroup) over the Riemann sphere $\P^1$. We will see that (up to a choice of a Cartan subgroup $H\subset B$) in each of these bundles we may define a set of sections, that we call {\it fundamental}, that are in one-to-one correspondence with the Weyl group of $G$ and that, under this correspondence, reflections of the root system correspond to $\P^1$-bundles containing pairs of these sections. Moreover, we will show that one of these sections is minimal --in a deformation theoretical sense, see Definition \ref{defn:sections}-- and that the $G/B$-bundle is determined by this minimal section and by its self-intersection numbers within the $\rk(G)$ $\P^1$-bundles containing it. This information may be then represented by what we call a {\it tagged Dynkin diagram} (see Theorem \ref{thm:intergrothen} for a precise statement). 

As an application, we prove, in Section \ref{sec:appCP}, the following statement:

\begin{theorem}\label{thm:RHVMRT} Let $X$ be a Fano manifold of Picard number one, and $p:\cU\to \cM$ be an unsplit dominating complete family of rational curves satisfying that 
the evaluation morphism $q:\cU\to X$ is smooth. Assume furthermore that the fiber $\cM_x=q^{-1}(x)$ is a rational homogeneous space for every $x\in X$. Then $X$ is rational homogeneous.
\end{theorem}

The relation of this statement with the Campana--Peternell conjecture comes from the fact that the smoothness of 
$q$ is satisfied by any unsplit family of rational curves in $X$ if we assume that $T_X$ is nef. 

We note also that this results resembles the main theorem in \cite{HH} (proven first by Mok for Hermitian symmetric spaces and homogeneous contact manifolds, see \cite{Mk3}), but there are certain differences between the two statements: on one hand Hong and Hwang need to assume that the image of $\cM_x$ into $\P(\Omega_{X,x})$, the so called {\it VMRT of $\cM$ at $x$}, is projectively equivalent to the VMRT of a rational homogeneous manifold $X'$, while we do not need to consider any particular projective embedding of $\cM_x$. On the other hand, they only need to check the above property on a general point $x\in X$, while we need to assume that $q$ is smooth, and that every $\cM_x$ is rational homogeneous. Note that if the relative Picard number $\rho(\cU/\cM)$ is one, it is enough to check the homogeneity condition for a general $x\in X$ (see \cite[Section 3]{Hw4}). 

Unfortunately, one cannot expect to characterize all the rational homogeneous manifolds of Picard number one in this way, since there are some examples  
for which $\cM_x$ is not homogeneous (\cite{LM}, see also \cite[Proposition 2.20]{MOSWW}). We expect that a similar treatment of non complete families of rational curves may lead to a solution of this problem.

Finally we note that Theorem \ref{thm:RHVMRT} depends on a characterization of complete flag manifolds which is a slightly stronger version of the main theorem in \cite{OSWW}, see Theorem \ref{thm:newOSWW}. We show how to prove this statement in Appendix \ref{sec:OSWWrev}.

\medskip

\noindent{\bf Warning:} In previous papers (\cite{MOSW,MOSWW,OSWW}) we used the term ``smooth $\P^1$-fibration'' for any smooth morphism whose fibers are isomorphic to $\P^1$, reserving the name ``$\P^1$-bundle'' for the Grothendieck projectivization of a rank two vector bundle. In the present paper, which is about smooth rational homogeneous fiber bundles, 
the term ``$\P^1$-bundle'' refers to a smooth morphism whose fibers are isomorphic to $\P^1$.


\section{Rational homogeneous bundles}\label{sec:notn}

Throughout this paper all the varieties will be smooth, projective and defined over the field of complex numbers. A {\it (smooth, projective) fiber bundle over $X$} is a (smooth, projective) morphism $q:E\to X$ between smooth varieties. We will be interested in the case in which $q$ is smooth, projective, and also {\it isotrivial}, which means that all the fibers of $q$ over closed points are isomorphic; if $F$ is a variety isomorphic to the fibers of 
$q$, we will say that the fiber bundle is an {\it $F$-bundle}. According to a classical theorem of Fischer and Grauert (cf. \cite[page~29]{BPVV}), it follows that a smooth projective $F$-bundle $q$ is {\it locally trivial} in the analytic topology: there exists then an open covering $\{U_i,\,i\in I\}$ of $X$ with the analytic topology and isomorphisms $\phi_i:U_i\times F\to q^{-1}(U_i)$ commuting with the corresponding projections onto $U_i$, named {\it trivializations of $q$}. We may consider $\{\phi_i\}$ to be the whole atlas of trivializations of $q$, and then $q$ may be reconstructed from the complex varieties $U_i\times F$ by means of the transitions $$\phi_{ij}=\phi_j^{-1}\circ\phi_i: U_{ij}\times F\longrightarrow U_{ij}\times F,\quad\mbox{(}U_{ij}:=U_i\cap U_j\mbox{)},$$ or, equivalently by the corresponding maps
$\theta_{ij}:U_{ij}\to \Aut(F)$ satisfying $\phi_{ij}(x,y)=(x,\theta_{ij}(x)(y))$. It turns out that the $\phi_{ij}$ define a cocycle $\theta\in H^1(X,\Aut(F))$, which completely determines the $F$-bundle $q:E\to X$. Note that this cocycle is defined over the analytic space associated to $X$, that we denote also by $X$, by abuse of notation.

We will assume also that every fiber $F$ of $q$ is {\it rational homogeneous}, i.e., that it is isomorphic to the quotient of a semisimple group by a parabolic subgroup. 

\begin{remark}\label{rem:unique}
Note that the representation of $F$ as a quotient of a semisimple group is not unique. Besides eventual choice of two semisimple groups with the same universal cover, in some occasions a rational homogeneous variety $F$ may be written as the quotient of two semisimple groups with different associated Lie algebras. 
This is the case  of $\P^{2n-1}$ (that can be written as a quotient of $\SL(2n)$, $\PGL(2n-1)$ --these two with the same Lie algebra-- and $\SP(2n)$), of the quadric $Q^5$ (quotient of $\SO_7$ and $\mathrm{G}_2$), and the spinor variety parametrizing $r-1$ dimensional linear subspaces in a quadric of dimension $2r-1$ (quotient of $\SO_{2r+1}$ and $\SO_{2r+2}$). For irreducible rational homogeneous manifolds (i.e. rational homogeneous manifolds that are not products), this is the complete list (see \cite[Ch.~3, p.~75]{Akh}). 
\end{remark}

However, it is known that the identity component of the automorphism group  $\Aut(F)$ is semisimple \cite[Thm.~3.11]{Huck}, and it acts transitively on $F$. We will denote it by $G_F$, or $G$ if there is no possible confusion, and then we may write $F$ as a quotient $G/P$ for some parabolic subgroup $P\subset G$. If we further assume that the base $X$ is simply connected (this is the case, for instance, if $X$ is Fano or, more generally, rationally connected), then it follows that the cocycle $\theta\in H^1(X,\Aut(F))$ lies in $H^1(X,G)$; otherwise, since $G$ is a normal subgroup of $\Aut(F)$, $\theta$ would define a nonzero cocycle in $H^1(X,\Aut(F)/G)$, which in turn defines an \'etale covering of $X$, a contradiction. 

Assume now that $q:E\to X$ is an $F$-bundle defined by a cocycle $\theta\in H^1(X,G)$, for a rational homogeneous manifold $F$ whose automorphism group has identity component $G$. Then $\theta$ defines a $G$-bundle $q_G:E_G\to X$, given by the glueing of $U_i\times G$ by means of the transition functions:
$$
\theta_{ij}:U_{ij}\to G\hookrightarrow \Aut(G),
$$
where $G$ embeds in $\Aut(G)$ as the set of automorphism given by left multiplication with elements of $G$. There is then a well defined action of $G$ on $E_G$, given by right multiplication by elements of $G$ on every chart. We may now say that the $G$-bundle $E_G$ is a {\it principal} $G$-bundle over $X$. If we now take any parabolic subgroup ${P'}\subset G$, we may consider the $G/{P'}$-bundle: 
$$
q_{P'}:E_{P'}:=E_G\times_GG/{P'}:=(E_G\times G/{P'})/\!\!\sim_G\,\,\to X,
$$
where $\sim_G$ is the equivalence relation defined by $(x,g{P'})\sim_G (xh,h^{-1}g{P'})$ for every $h\in G$. One may easily check this $G/{P'}$-bundle is defined also by the same cocycle $\theta\in H^1(X,G)$. Note that $H^1(X,G)$ is the cohomology of $X$, regarded as an analytic space, with values in $G$, 
so $E_{P'}$ is, a priori, only defined as a compact analytic space. However, since we can embed $E_{P'}$ into the projectivization of the holomorphic vector bundle ${q_{P'}}_*(\omega^{-r}_{E_{P'}|X})\cong E_G\times_{P'}H^0(G/P',\omega^{-r}_{G/{P'}})$ (for $r\gg 0$), which is an algebraic vector bundle by the GAGA principle (\cite{GAGA}), it follows that $E_{P'}$ is indeed a projective variety. 

Furthermore, given two parabolic subgroups $P_1\subset P_2\subset G$, one has a natural map $q_{P_1,P_2}:E_{P_1}\to E_{P_2}$ satisfying $q_{P_2}\circ q_{P_1,P_2}=q_{P_1}$. In particular we will consider a Borel subgroup $B$ of $G$ contained in $P$, and define $\tl{E}:=E_B$, with projection map $\tl{q}=q_B:\tl{E}\to X$, and, for every parabolic ${P'}\supset B$, a map $\tl{q}_{P'}:=q_{B,{P'}}:\tl{E}\to E_{P'}$. In particular we have a projection $\tl{q}_P:\tl{E}\to E_P=E$, that we will denote by $\pi:\tl{E}\to E$. The bundle $\tl{q}:\tl{E}\to X$ will be called the {\it flag bundle} associated to $q:E\to X$. Note that $\pi$ is also a flag bundle over $E$, whose fiber we write as $G'/B'$, where $G'$ is the the semisimple group obtained by subsequently quotienting $P$ by its unipotent subgroup and then by the center, and $B'$ is the image of $B$ in this quotient.

We will denote by $k$ the {\it rank of the group }$G$, which can be defined as the dimension of a maximal torus contained in $B$, or as the number of simple roots of the root system of the Lie algebra $\fg$ of $G$, or as the number of minimal parabolic subgroups of $G$ containing $B$. Note that, denoting by $P_1,\dots, P_k$ these minimal parabolic subgroups, the morphisms $\pi_i:=\tl{q}_{P_i}:\tl{E}\to E_i:=E_{P_i}$ are $\P^1$-bundles.

For the reader's convenience, we present the following diagram that illustrates the notation used (corresponding to a minimal parabolic subgroup $P_i\subset P$):
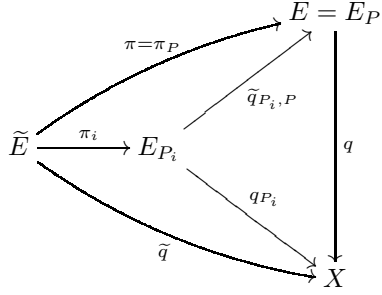
$$
\xymatrix@=35pt{&&E=E_P\ar[dd]^{q}\\
\tl{E}\ar@/^{3mm}/[rru]^{\pi=\pi_P}\ar[r]^{\pi_i}\ar@/_{3mm}/[rrd]_{\tl{q}}&E_{P_i}\ar[ru]_{\tl{q}_{P_i,P}}\ar[rd]^{q_{P_i}}&\\
&&X
}
$$

Finally, we may describe the relative canonical class of a $G/B$-bundle as above. It is well known that $K_{G/B}$ may be described as follows. Consider $\Delta:=\{\alpha_i,\,\,i=1,\dots, k\}$ a base of simple roots of the root system $\Phi$ of the Lie algebra $\fg$ of $G$, so that every positive root $\beta\in\Phi^+$ is a nonnegative integer combination of the $\alpha_i$'s, and define the integers $b_1,\dots,b_k$ by the formula:
\begin{equation}\label{eq:posroots}
\sum_{\beta\in\Phi^+}\beta=\sum_{t=1}^kb_t\alpha_t.
\end{equation}
Consider also the relative canonical divisors $K_{t}$ of the projections $G/B\to G/P_t$. Then:
$$
K_{G/B}=\sum_{t=1}^kb_tK_t.
$$
Table \ref{magicnumbers} below contains the values of the integers $b_1,\dots,b_k$, for $G$ simple (the numbering of the nodes corresponds to the one  in \cite[p. 58]{Hum2}).
\renewcommand{\arraystretch}{1.15} 


\begin{table}[!ht]
\centering
\begin{tabular}{|c|c|}
\hline
 $\cD$& $b_1,\dots,b_k$  \\ \hline \hline
 ${\rm A}_k$& $k,(k-1)2, \dots, 2(k-1), k$\\ \hline
 ${\rm B}_k$& $2k-1,2(2k-2),\dots,(k-1)(k+1),k^2$\\ \hline
 ${\rm C}_k$& $2k,2(2k-1),\dots,(k-1)(k+2),k(k+1)/2$\\ \hline
  ${\rm D}_k$& $2k-2,2(2k-3),\dots,(k-2)(k+1),(k-1)k/2,(k-1)k/2$\\ \hline
 ${\rm E}_6$& $16,22,30,42,30,16$\\ \hline
 ${\rm E}_7$& $34,49,66,96,75,52,27$\\ \hline
 ${\rm E}_8$& $92, 136,182,270,220,168,114,58$\\ \hline
 ${\rm F}_4$& $16,30,42,22$\\ \hline
 ${\rm G}_2$& $10,6$\\ \hline
\end{tabular}
\caption[VMRTs of RH]{Coefficients of $K_{G/B}$ in terms of the relative canonical divisors $K_t$.} \label{magicnumbers}
\end{table}

On the other hand, one may show that $T_{G/B}$ may be constructed upon the relative tangent bundles of the projections $G/B\to G/P_t$, via successive use of Lie brackets. We refer the reader to \cite[Construction 1]{MOSW} for details on this construction. In our situation, we may show that, identifying $\pi_t |_{{\tl{q}}\,^{-1}(x)}:\tl{q}\,^{-1}(x) \to {\tl{q}}\,_t^{-1}(x)$ with the projection $G/B\to G/P_t$, for any $x \in X$ then
\begin{equation}\label{eq:rel}
\cO_{\tl E}(K_{\pi_t})|_{\tl{q}^{-1}(x)}= \cO_{G/B}(K_t),
\end{equation}
and the corresponding successive use of Lie brackets produces the relative tangent bundle of ${\tl{E}}$. In particular, one obtains:

\begin{lemma}\label{lem:relcan}
The relative canonical class of $\tl{q}$ satisfies:
 $$
 K_{\tl{q}}=\sum_{t=1}^kb_tK_{\pi_{t}}.
 $$ 
 
\end{lemma}

The same argument allows us to compute (the pull-back to $\tl{E}$ of) the relative canonical class of $q:E\to X$. Set $D=\{1, \dots, k\}$; let  $I \subset D$ be the subset defining  the parabolic subgroup $P$, i.e. the set of indices such that the Lie subalgebra $\fp\subset\fg$ of $P$ satisfies:
\begin{equation}
\fp=\fh\oplus\bigoplus_{\alpha\in\Phi^+} \fg_{-\alpha}\oplus\bigoplus_{\alpha\in\Phi^+(I)} \fg_\alpha,
\label{eq:cartanparab}
\end{equation}
where $\Phi^+(I)$ denotes the subset of $\Phi^+$ generated by the simple roots  $\{\alpha_i,\,\,i\notin I\}$. Then, arguing as before, we have:

\begin{lemma}\label{lem:relcan2}
With the same notation as above, let us consider the integers $c_1,\dots,c_k$ determined by:
$$
\sum_{\beta\in\Phi^+(I)}\beta=\sum_{t=1}^kc_t\alpha_t.
$$ 
Then:
\begin{equation}\label{eq:canrel2}
\pi^*K_{q}=\sum_{t=1}^k(b_t-c_t)K_{\pi_{t}}.
\end{equation}
\end{lemma}

\begin{proof}
It is enough to note $\pi^*K_{q}$ equals the difference $K_{\tl{q}}-K_{\pi}$, where the first summand has been computed in Lemma \ref{lem:relcan}, and the second can be computed by applying the same lemma to the flag bundle
$\pi:\tl{E}\to E$.
\end{proof}


\section{Flag bundles over $\P^1$: remarks on Grothendieck's theorem}\label{sec:groth}

In this section we will study rational homogeneous bundles over $\P^1$. More concretely, since all of them are contractions of flag bundles, we will concentrate on the case of a $G/B$-bundle $\tl{q}:\tl E\to \P^1$. We will use the notation introduced in Section \ref{sec:notn}. 

Inspired by the case of projective bundles over $\P^1$, that are determined by the restriction of their relative tangent bundles to a minimal section, we will study sections of $\tl{\pi}$ that satisfy the following minimality condition:

\begin{definition}\label{defn:sections}
A section $s:\P^1\to\tl{E}$ of $\tl{q}$ is called {\it minimal} if, for every $x\in\P^1$, the irreducible component $H$ of the scheme $\Hom_{\P^1}(\P^1,\tl{E};x, s(x))$ parametrizing sections of $\tl{q}$ sending $x$ to $s(x)$ is zero-dimensional.
\end{definition}

The existence of minimal sections on the $G/B$-bundle $\tl{E}$ is an immediate consequence of the Bend and Break Lemma (see for instance \cite[Proposition 3.11]{De}), once we prove the existence of a section of any kind; this, in turn, may be obtained, either from a much more general result of Graber, Harris and Starr, \cite{GHS}, or by means of ad hoc arguments for each of the possible types of semisimple group $G$, or (as we will do later) by means of Grothendieck's theorem on principal $G$-bundles. 

\begin{definition}\label{def:}
Given $\sigma:\P^1\to\tl{E}$ a section of $\tl{q}:\tl{E}\to\P^1$, we may consider the projections ${\pi}_{t}:\tl{E}\to E_{t}$, where $P_t$, $t=1,\dots,k$.  As noted before, the ${\pi}_{t}$'s are $\P^1$-bundles, and we may consider the surfaces $\tl{E}_t$ appearing as the fiber products of $\tl{E}$ and the section $\sigma$, for every $t$:
$$
\xymatrix@=35pt{\tl{E}_t\ar[r]\ar[d]^{\pi_{t}}&\tl{E}\ar[d]^{{\pi}_{t}} \\
\P^1\ar@/^/[u]^{\sigma_t} \ar[ru]^{\sigma}\ar[r]_{{\pi}_{t}\circ \sigma}&\P^1}
$$
Then, for every $t=1,\dots,k$, $\tl{E}_t$ is a $\P^1$-bundle, that we call the {\it pivotal} $\P^1$-bundles of $\tl{q}$ with respect to $\sigma$.
\end{definition}

Note that the section $\sigma$ defines a section $\sigma_t:\P^1\to \tl{E}_t$. In the case in which $\sigma$ is minimal in $\tl{E}$, then it is also minimal in $\tl{E}_t$. In particular the integers $d_t=K_{\pi_t} \cdot \sigma(\P^1)$ are non negative. 
The main goal of this section is to show that a minimal section $\sigma$, together with the non negative integers $d_t$, determine completely the $G/B$-bundle $\tl{E}$, and the main tool we will use is Grothendieck's theorem on principal $G$-bundles over $\P^1$.

\subsection{Grothendieck's theorem}\label{ssec:grothen}

Let $G$ be a semisimple Lie group, and $H\subset G$ be a Cartan subgroup, which is a maximal abelian subgroup of $G$. Let $N$ denote the normalizer of $H$ in $G$, whose quotient $N/H$ is called the {\it Weyl group of $G$ with respect to $H$}. Given any smooth complex variety $X$ we may consider the cohomology of $X$ with values in the sheafified groups $H$ and $G$ (that, by abuse of notation, we denote with the same symbols), and the corresponding natural map $H^1(X,H)\to H^1(X,G)$. This map is equivariant with respect to the conjugation action of the groups $N$ and $H$. Since the latter acts trivially on both sets, and $N\subset G$ acts trivially on the second, we finally have a map:
\begin{equation}\label{eq:grothen}
H^1(X,H)/W\to H^1(X,G)
\end{equation}

Grothendieck's theorem says the following:

\begin{theorem}[Grothendieck's theorem (cf. \cite{Gro1})]\label{thm:groth}
With the same notation as above, for $X=\P^1$, the map (\ref{eq:grothen}) is a bijection. 
\end{theorem}

This theorem was originally  stated for the case of a reductive group $G$, but we will stick to the semisimple case, which is the one we are interested in.

\begin{remark}\label{rem:GrothP1}
The standard geometric interpretation of this theorem concerns the case $G=\PGL(r+1)$, where the theorem tells us that any $\P^r$-bundle over $\P^1$ is the projectivization of a direct sum of line bundles. For instance, for $r=1$ and a cocycle $\theta\in H^1(\P^1,G)$, any choice of a Cartan subgroup $H\subset G$ corresponds to a choice of two sections $C_0$ and $C_\infty$ of the corresponding $\P^1$-bundle $\tl{q}:\tl{E}\to \P^1$, satisfying that $C_0\cdot C_\infty=0$. These two sections correspond to two quotients $\cE\to\cO_{\P^1}(a_i)$, $i=1,2$, where $\tl{E}$ is isomorphic to the projectivization of a bundle $\cE\cong\cO_{\P^1}(a_1)\oplus\cO_{\P^1}(a_2)$. Typically $C_0$ is associated with the quotient to  $\cO_{\P^1}(a_1)$, when we choose $a_1\leq a_2$. Note that $C_0$ is a minimal section of the $\P^1$-bundle, in the sense of Definition \ref{defn:sections}.
\end{remark}

\begin{setup}\label{set:groth}
Throughout the rest of Section \ref{sec:groth}, $G$ will denote a semisimple Lie group, $H\subset B\subset G$ a pair of Cartan and Borel subgroups, and $\theta\in H^1(\P^1,H)$ a cocycle, whose image into $H^1(\P^1,G)$ defines a $G/B$-bundle $\tl{q}:\tl{E}\to X$.
\end{setup}
\begin{remark}\label{rem:Borelsec} 
Note that all the possible Borel subgroups of $G$ are conjugate, and the choice of different Borel subgroups produces isomorphic $G/B$-bundles. On the other hand, the choice of a preimage of $\theta$ into $H^1(\P^1,B)$ via the natural map $H^1(\P^1,B)\to H^1(\P^1,G)$ (in the standard language, the choice of a {\it reduction of $E_G$ via the inclusion $B\hookrightarrow G$}), is equivalent to the choice of a holomorphic section of $\tl{q}:\tl{E}\to X$.  
\end{remark}

\begin{remark}\label{rem:chambers}
 Following \cite{Gro1}, the kernel $L$ of the exponential map sending every $v\in\fh$ to $\exp(2\pi \imath v)\in H$ may be identified with the dual lattice of the lattice in $\fh^\vee$ generated by the root system $\Phi$. Moreover we have natural homomorphisms of groups:
 \begin{equation}\label{eq:exp}
 H^1(\P^1,H)\cong H^2(\P^1,L)\cong L\subset \fh,
 \end{equation}
 satisfying that the conjugation action of $W$ in $H^1(\P^1,H)$ corresponds to the restriction of the dual action of $W$ on $\fh$, as the Weyl group of $\Phi$. 
\end{remark}

\subsection{Fundamental sections of a flag bundle}\label{ssec:fundsec}

Let us start by noting some general facts on parabolic subgroups, and their Levi decompositions:

\begin{remark}\label{rem:Borels1}
Given a semisimple Lie group $G$ and a Cartan subgroup $H$, let us denote with gothic fonts $\fh,\fg$, as usual, their associated Lie algebras. Then $\fh\subset\fg$ defines a root system $\Phi\subset\fh^\vee$. Any choice of a base of positive simple roots $\Delta=\{\alpha_t,\,\,\,t=1,\dots,k\}$ corresponds to a choice of a fundamental chamber of the action of the Weyl group $W$ on $\fh$, $\cP=\{v\in\fh |\,\,\alpha_t\cdot v\geq 0, \,\,\forall t\}\subset\fh$, to a choice of a decomposition $\Phi=\Phi^+\cup\Phi^-$ of the root system in two sets of positive and negative roots, and to a choice of a Borel subalgebra $\fb:=\fh\oplus\bigoplus_{\beta\in\Phi^-}\fg_\beta$ and a corresponding Borel subgroup $B\subset G$ containing $H$ (\cite[27.3]{Hum1}).
Given such a Borel subgroup $B$, for every $t=1,\dots, k$ we denote by $P_t\supset B$ the parabolic subgroups of $G$ whose Lie algebras are $\fb\oplus\fg_{\alpha_t}$. Dividing $P_t$ by its unipotent radical first, and then by its center, we obtain a semisimple group $G_t$, isomorphic to $\PGL(2)$,
satisfying that the images $H_t$ of $H$, and $B_t$ of $B$, into $G_t$, are respectively a Cartan and a Borel subgroup of $G_t$. Note that we have an isomorphism $G_t/B_t\cong P_t/B$, and the natural map $P_t/B\subset G/B$ allows us to identify $G_t/B_t$ with the fiber of the canonical projection $G/B\to G/P_t$ passing by the class of the identity modulo $B$. Moreover, the induced homomorphism of Lie groups 
\begin{equation}\label{eq:degrees1}
\phi:H\to\prod_t H_t
\end{equation} 
is finite and surjective, since one may easily check that its differential at the identity gives an isomorphism of Lie algebras $d\phi_e:\fh\to\bigoplus_t \fh_t$.
\end{remark}

\begin{remark}\label{rem:Borels2}
Conversely, if we start from a semisimple group $G$ and a Borel subgroup $B$, we may consider the minimal parabolic subgroups of $G$ containing $B$, $P_1,\dots, P_k$. Considering, as above, the semisimple part $G_t$ of $P_t$, we have a natural surjective map $B\to \prod_tB_t$. Given, for every $t$, a Cartan subgroup $H_t\subset B_t$, the inverse image of $\prod_tH_t$ is a Cartan subgroup of $G$ contained in $B$.
\end{remark}

\begin{example}\label{rem:latticeP1}
In the case of $G=\PGL(2)$, each of the two possible choices of a basis of positive simple roots $\Delta=\{\alpha\}$ provides an isomorphism $\alpha:\fh\to \C$, under which the lattice $L$ defined in Remark \ref{rem:chambers} may be identified with $\Z\subset\C$. Now, if we start from a cocycle $\theta\in H^1(\P^1,H)$, we may choose $\Delta=\{\alpha\}$ so that the image of $\theta$ in $L$  via (\ref{eq:exp}) has nonnegative $\alpha$-coordinate. This integer is classically known as the {\it invariant} of the $\P^1$-bundle defined by $\theta$, which is the intersection of its minimal section $C_0$ with its relative canonical line bundle.
\end{example}
 
In our general setting, we may choose the fundamental chamber $\cP\subset\fh$ so that the image of $\theta\in H^1(\P^1,H)$ into $\fh$ belongs to $\cP$ and, by Remark \ref{rem:Borelsec}, we may assume, without loss of generality, that $B$ is the Borel subgroup of $G$ containing $H$ corresponding to this choice of $\cP$. The corresponding base of simple roots is denoted by $\Delta=\{\alpha_1,\dots,\alpha_k\}$. Any other Borel subgroup of $G$ containing $H$ is the conjugation of $B$ by an element $w\in W$, hence it will be denoted by $B_w$.
 
 \begin{definition}\label{def:fundsec}
 In the setup of \ref{set:groth}, with the same notation as above the image of $\theta\in H^1(\P^1,H)$ into each of the sets $H^1(\P^1,B_w)$, $w\in W$, defines a section of the $G/B$-bundle $\tl{q}:\tl{E}\to \P^1$, that we denote by $\sigma_w:\P^1\to C_w\subset \tl{E}$. These sections are called the {\it fundamental sections of $\tl{E}$ with respect to $\theta$}. In particular, the section corresponding to $B$ will be called the {\it minimal fundamental section} of $\tl{E}$ with respect to $\theta$, and denoted by $\sigma:\P^1\to C\subset \tl{E}$. 
 \end{definition}
 
 \begin{remark}\label{rem:chamberssections}
Note that the set of fundamental sections of $\tl{E}$ is in one to one correspondence with the set of possible Weyl chambers of $W$ in $\fh$, which is known to be bijective to $W$.
 \end{remark}

For instance, in the case of a $\P^1$-bundle treated in Remark \ref{rem:GrothP1}, given $H\subset G$ we have precisely two Borel subgroups containing it, each of them giving rise to one the (fundamental) sections $C_0$ or $C_\infty$. In our general setup, we may consider, for every Borel subgroup $B_w$ containing $H$, and each positive simple root $\alpha_t\in\Delta$, the pivotal $\P^1$-bundle $\tl{E}_{w,t}:=({\pi}_t\circ \sigma_w)^*\tl{E}$, fitting in the cartesian square: 
$$
\xymatrix@=35pt{\tl{E}_{w,t}\ar[r]\ar[d]^{{\pi}_{w,t}}&\tl{E}\ar[d]^{{\pi}_t} \\
\P^1\ar@/^/[u]^{\sigma_{w,t} }
\ar[r]_{{\pi}_t\circ \sigma_w} \ar[ru]_{\sigma_w}&{E}_t}
$$
Let us denote by $C_{w}$ the image of the section $\sigma_w$. On the other hand, we may consider the reflection $r_t\in W$ corresponding to a positive simple root $\alpha_t$, and the section $\sigma_{r_t w}:\P^1\to \tl{E}_{r_t w,t}$, with image $C_{r_t w}$. Then we may claim the following, that follows from \cite[29.3 Lemma B]{Hum1}:

\begin{proposition}\label{prop:cobord}
With the same notation as above, for every $w\in W$ and every $\alpha_t\in\Delta$, $$\tl{E}_{w,t}=\tl{E}_{r_t w,t},$$ and the sections $C_{w}$ and $C_{r_t w}$ are the $C_0$ and $C_\infty$ sections of the $\P^1$-bundle $\tl{E}_{w,t}$.
\end{proposition}

\begin{remark}\label{rem:whoismin}
The correspondence  between $\{C_{w},C_{r_t w}\}$ and $\{C_0,C_\infty\}$ depends on the particular choices of $t$ and $w$, and may be expressed as follows: $C_w$ is the $C_0$ section of $\tl{E}_{w,t}$ if and only if $w(\alpha_t)$ is a positive root. 
In the particular case of $B_w=B$, it follows that the corresponding minimal fundamental section $C$ is a minimal (fundamental) section for every pivotal $\P^1$-bundle $\tl{E}_{t}$ containing it.
\end{remark}

\begin{remark}\label{rem:P1cocycles}
Note also that each pivotal $\P^1$-bundle $\tl{E}_{t}$ is defined by the image $\theta_t$ of the cocycle $\theta$ via the natural map $H^1(\P^1,H)\to H^1(\P^1,H_t)$, induced by the group homomorphism $H\to H_t$ defined in Remark \ref{rem:Borels1}. 
\end{remark}

\subsection{The tag of a flag bundle}\label{ssec:invariant}

Let us consider now the differential at the identity of the homomorphism $\phi$ of complex tori (\ref{eq:degrees1}), together with the corresponding exponential maps defined in Remark \ref{rem:chambers}. We have a commutative diagram of group homomorphisms, with exact rows: 
\begin{equation}\label{eq:degrees2}
\xymatrix@=45pt{L\ar[d]\ar[r]&\fh\ar[r]^{\exp(2\pi \imath \,\underline{\hspace{0.15cm}})}\ar[d]^{\phi_{*e}}&H\ar[d]^{\phi} \\
\bigoplus_t L_t \ar[r]&\bigoplus_t\fh_t
\ar[r]_{\exp(2\pi \imath \,\underline{\hspace{0.15cm}})} &\prod_tH_t}
\end{equation}
where $L_t$ denotes the kernel lattice for each $H_t$.

\begin{definition}\label{def:invariant}
With the same notation as above, denoting by $\alpha'_t\in\fh_t^*$ the positive root associated to the Borel subgroup $B_t$, for each $t$, the homomorphism of lattices:
$$
\xymatrix@=35pt{L\ar[r]_{\phi_{*e}}\ar@/^{4mm}/[rr]^{\delta}&\bigoplus_t L_t \ar[r]_{\oplus_t\alpha'_t} &\Z^k}
$$ 
is called the {\it tagging map} of $G$, and the image of $\theta\in H^1(\P^1,H)\cong L$ via $\delta$, $\delta(\theta)$ is called the {\it tag of the $G/B$-bundle $\tl{q}:\tl{E}\to \P^1$ defined by $\theta$}.  \end{definition}

\begin{proposition}\label{prop:grothunique}
With the same notation as above, the $G/B$-bundle $\tl{q}:\tl{E}\to \P^1$ defined by $\theta$ is completely determined by its tag $\delta(\theta)=(d_1,\dots,d_k)$. Moreover, for every $t$, $d_t=K_{{\pi}_t}\cdot C$, being $C$ the minimal fundamental section of $\tl{E}$ introduced in Definition \ref{def:fundsec}.
\end{proposition}

\begin{proof}
For the first part it is enough to check that the map $L\to\bigoplus_t L_t$ is injective, and this follows from the fact that $\phi_{*e}:\fh\to \bigoplus_t\fh_t$ is an isomorphism. For the second, note that $d_t$ equals the invariant of the corresponding pivotal $\P^1$-bundle $\tl{E}_t$, hence it may be interpreted as the intersection number of the relative canonical divisor $K_{{\pi}_t}$ with the minimal fundamental section $C$. 
\end{proof}

\begin{remark}\label{rem:invariants}
As a consequence of Proposition \ref{prop:grothunique}, the tag of any $G/B$-bundle over $\P^1$ with respect to a cocycle $\theta$ lies in $\Z_{\geq 0}^k$. 
Then $\delta(\theta)$ shall be thought of as the choice of a non negative integer $d_t$ for every positive simple root of $G$ (with respect to the choice of a base of simple roots, as in Remark \ref{rem:Borels1}). In particular it is natural to represent it as a {\it tagged Dynkin diagram}, i.e. by ``tagging'' the $t$-th node of the Dynkin diagram, corresponding to the root $\alpha_t$, with the non negative integer $d_t$. A tagged Dynkin diagram obtained in this way is called {\it admissible with respect to} $G$. Note that a flag manifold may be written as the quotient of different semisimple groups with the same Dynkin diagrams, for which the concept of admissibility is not the same, in general. Hence it is not true that, for a $G/B$-bundle over $\P^1$, any $k$-tuple of non negative integers is admissible with respect to $G$. \end{remark}

\begin{example}
For instance, a $\P^n$-bundle $\P(\bigoplus_i\cO_{\P^1}(a_i))$, with $a_0\leq\dots,\leq a_n$, is determined by the tag $(a_1-a_0,\dots,a_n-a_{n-1})$ on the Dynkin diagram ${\rm A}_n$ (with the standard numbering of its nodes). Hence, every tag is admissible for the group $\PGL(n+1)$, whereas for the group $G=\SL(n+1)$, the lattice $\delta(L)$ of admissible tags has index $n$ in $\Z^n$.
\end{example}

\subsection{Minimal fundamental sections vs Minimal sections}\label{ssec:uniqueness}

Finally we will show the relation between the minimal fundamental sections of a $G/B$-bundle and its minimal sections. 

\begin{lemma}\label{lem:minismin1}
With the same notation as above, the minimal fundamental section $C$ of a $G/B$-bundle defined by a cocycle $\theta\in H^1(\P^1,H)$ is a minimal section in the sense of Definition \ref{defn:sections}.
\end{lemma}

\begin{proof}
Since, by construction, $C$ has degree greater than or equal to zero with respect to the relative canonical divisors $K_{{\pi}_t}$, we may claim that the relative tangent bundle of $\tl{E}$ over $C$, that can be obtained (cf. \cite[Construction 1]{MOSW}) from the relative canonical bundles $\cO_{\tl{E}}(K_{{\pi}_t})$ by means of successive Lie brackets, has no positive summands, and hence $C$ has no deformations with a point fixed.
\end{proof}

Conversely, we will show now that a minimal section of a $G/B$-bundle is a minimal fundamental section, for a certain choice of a Cartan subgroup. In other words, minimal sections are not necessarily unique, but they are determined by the different choices of Cartan subgroups of $G$. 

\begin{proposition}\label{prop:minismin2}
Let $\tl{q}:\tl{E}\to \P^1$ be a $G/B$-bundle over $\P^1$, and let $\sigma':\P^1\to C'\subset \tl{E}$ be a minimal section in the sense of Definition \ref{defn:sections}. Then there exists a Cartan subgroup $H'\subset G$ such that $C'$ is a minimal fundamental section with respect to $H'$.
\end{proposition}

\begin{proof}
The section $\sigma'$ provides an inverse image $\theta'$ in $H^1(X,B)$ of the cocycle in $H^1(X,G)$ defining $\tl{E}$. With the same notation as in Remark \ref{rem:Borels1}, we consider the images $\theta'_t$ of $\theta'$ into $H^1(\P^1,B_t)$, $t=1,\dots,k$, defined by the minimal parabolic subgroups $P_t$ containing $B$; they define $k$ pivotal $\P^1$-bundles $\tl{E}'_t$ containing $C'$ as a section. Since $C'$ is minimal in $\tl{E}$ in the sense of Definition \ref{defn:sections}, it is also minimal in each $\tl{E}'_t$. But on a $\P^1$-bundle it is clear that a minimal section is a minimal fundamental section, from what it follows that $\theta'_t\in H^1(\P^1,B_t)$ has a (unique) inverse image in $H^1(\P^1,H'_t)$, for certain Cartan subgroups $H'_t$, $t=1,\dots, k$. By Remark \ref{rem:Borels2}, the inverse image of $\prod_tH'_t$ into $B$ is a Cartan subgroup $H'$, and the $k$-tuple $(\theta'_1,\dots,\theta'_k)$ has a unique inverse image $\theta''\in H^1(\P^1,H')$, which maps into $\theta'\in H^1(\P^1,B)$. We conclude by noting that, by construction, $C'$ is a minimal fundamental section of $\tl{E}$ with respect to $H'$.
\end{proof}

\begin{remark}\label{rem:invariant}
Given a $G/B$-bundle $\tl{q}$, we have defined its tag with respect to a choice of Cartan subgroup $H\subset G$ and a cocycle $H^1(\P^1,H)$. But two of these possible choices, say $\theta\in H^1(\P^1,H)$ and $\theta'\in H^1(\P^1,H')$, are related by conjugation with an element $g\in G$, $H'=gHg^ {-1}$.  We may then consider a holomorphic map $f:\C \to G$ satisfying $f(0)=e$ and $f(1)=g$. We choose $B$ and the minimal section $C$ with respect to $H$ as above, and consider, for every $z\in \C$, the corresponding cocycle $f(z)\theta f(z)^{-1}\in H^1(\P^1,f(z)H f(z)^{-1})$ and its image into $H^1(\P^1,f(z)B f(z)^{-1})$, that defines the minimal fundamental section of $\tl{q}$ with respect to $f(z)H f(z)^{-1}$. We conclude that any two minimal fundamental sections of $\tl{q}$ are numerically equivalent, hence $\delta(\theta)$ depends only on the isomorphism class of the $G/B$-bundle. 
\end{remark}

Summing up, we may write the following statement:

\begin{theorem}\label{thm:intergrothen}
Let $G$ be a semisimple Lie group, and $B\subset G$ be a Borel subgroup of $G$. 
For any admissible tagged Dynkin diagram $\cD$ there exist a $G/B$-bundle $\tl{q}:\tl{E}\to \P^1$, unique up to isomorphism, satisfying that the self-intersection of a minimal section $C$ of $\tl{q}$ in the $t$-th pivotal $\P^1$-bundle associated to $C$ is equal to the tag $d_t$ of the corresponding $t$-th node of $\cD$. 
\end{theorem}


\section{An application to the Campana--Peternell Conjecture}\label{sec:appCP}

Along this section, which is devoted to the proof of Theorem \ref{thm:RHVMRT}, $X$ will be a {\it CP-manifold}, that is a Fano manifold with nef tangent bundle. We will consider an unsplit complete family $\cM\stackrel{p}{\longleftarrow}\cU\stackrel{q}{\longrightarrow}X$ of rational curves in $X$. By definition, a {\it complete family of rational curves} is the universal family of curves parametrized by an irreducible component $\cM$ of $\rat^n(X)$. Then, the condition on the nefness of $T_X$ tells us that the evaluation morphism $q$ is dominant and smooth (see, for instance, \cite[Prop. 2.10]{MOSWW}).  For any $x\in X$, the smooth subvariety $q^{-1}(x)$, which is assumed to be a rational homogeneous manifold, will be denoted by $\cM_x$. 
Let us start by fixing some extra notation.

\begin{notation}\label{not:again}
We will denote by $G$ the identity component of $\Aut(\cM_x)$, so that every $\cM_x$ is isomorphic to the quotient of $G$ by a parabolic subgroup $P$. Then a Borel subgroup $B$ of $G$ contained in $P$ defines a $G/B$-bundle over $X$, that we denote by $\tl{q}:\tl{\cU}\to X$, that admits a contraction $\pi:\tl{\cU}\to \cU$. Note that $\pi$ is also a flag bundle over $\cU$, whose fiber we write as $G'/B'$ (as in Section \ref{sec:notn}, $G'$ is the semisimple group obtained by subsequently quotienting $P$ by its unipotent subgroup and then by the center, and $B'$ is the image of $B$ in this quotient).

We will denote by $k$ the rank of $G$, and set $D=\{1,\dots,k\}$. Every index $i \in D$ corresponds to a minimal parabolic subgroup $P_i\subset G$ containing $B$, that provides an elementary contraction ${\pi}_i:\tl{\cU}\to {\cU}_i$, which is a
a $\P^1$-bundle whose relative canonical classes we denote by $K_i$. We set $J:=\{i\in D\,\,|\,\,P_i\subset P\}$ and $I:=D\setminus J$, so that the contraction $\pi:\tl{\cU}\to \cU$ factors via ${\pi}_i$, for every $i\in J$. 
\end{notation}

Let $f:\P^1\to X$ be the normalization of any curve $\Gamma$ of the family $\cM$  and consider the pull-back $f^*\cU$; since $\Gamma$ is an element of the family  there is a natural section $s:\P^1\to s(\P^1) \subset f^*\cU$.  Denote by $ \tl s: \P^1\to s^*{f}^*\tl{\cU}$ a minimal section of the $G'/B'$-bundle $s^*{f}^*\tl{\cU}$ over $\P^1$ and  by $\tl{\Gamma}$ its image.  We have then the following commutative diagram:
$$
\xymatrix@=35pt{
& s^*{f}^*\tl{\cU}\pb \ar@{^{(}->}[]+<0ex,-2.5ex>;[d]\ar[r]_{\pi} 
\ar@{^{(}->}[]+<0ex,-2.5ex>;[d]
& \,\,\,\P^1 \ar@{=}[rd]  \ar@{^{(}->}[]+<0ex,-2.5ex>;[d]^s \ar@/_{2mm}/[]+<-2ex,+2.2ex>;[l]+<+0.7ex,+2.2ex>_{\tl s} &\\
& f^*\tl{\cU}\pb\ar[d]^{\tl f}\ar[r]_{\pi}& f^*\cU\pb\ar[d]^{f}\ar[r]_q& \P^1\ar[d]^{f}\\
\cM & \tl{\cU}\ar[l]^{p\circ\pi}\ar[r]_{\pi}&\cU 
\ar@/^{3mm}/[]+<-2ex,-2ex>;[ll]+<+2ex,-2ex>_{p} \ar[r]_q&X }
$$

The curve  $s(\P^1)$ is contracted by  the  composition map $f^*\cU\to \cU \to \cM$ and, moreover, it cannot deform in $f^*\cU$, otherwise its deformations would give a positive dimensional subvariety of $\cM$ determining the same curve in $X$. Hence, identifying $s^*{f}^*\tl{\cU}$ with its image in ${f}^*\tl{\cU}$, we may say that $\tl\Gamma$ is a minimal section of both the $G/B$-bundle $f^*\tl{\cU}$ and the $G'/B'$-bundle $s^*{f}^*\tl{\cU}$. By Proposition \ref{prop:minismin2} this minimal section can be regarded as a minimal fundamental section for each of the bundles, and Theorem \ref{thm:intergrothen} tells us that the two bundles are determined by two tagged Dynkin diagrams supported on the Dynkin diagrams $\cD$ and $\cD'$ of $G$ and $G'$, respectively. 

Moreover, $\cD'$ corresponds to the Dynkin subdiagram of $\cD$ obtained by eliminating from $\cD$ the nodes indexed by $I\subset D$ and, by the interpretation of the tags given in Proposition \ref{prop:grothunique}, the tags agree with the inclusion $\cD'\subset \cD$.

\begin{lemma}\label{lem:RHVMRT1}
With the same notation as above, the $G'/B'$-bundle $s^*{f}^*\tl{\cU}$ is trivial over $\P^1$. 
\end{lemma}

\begin{proof}
Recall that the $t$-th tag of $G/B$ is given by $d_t:=\tl{\Gamma}\cdot K_t$.
In view of Theorem \ref{thm:intergrothen} it is enough to show that $d_t=0$ for every $t \not \in  I$.
By Lemma \ref{lem:relcan2} applied to the $G/B$-bundle $f^*\tl \cU \to \P^1$:
\begin{equation}\label{eq:relcans}
\pi^*K_{q}=\sum_{t=1}^k(b_t-c_t)K_{\pi_{t}}.
\end{equation}
where the integers $b_t$, $c_t$ are given by 
\begin{equation}\label{eq:relcans2}
\sum_{\beta\in\Phi^+}\beta=\sum_{i\in D}b_i\alpha_i,\quad \sum_{\beta\in\Phi^+(I)}\beta=\sum_{j\in D}c_j\alpha_j=\sum_{j\in D\setminus I}c_j\alpha_j.
\end{equation}
Note that, by definition, we have $b_t \ge c_t$ for every $t \in D$.
By Equation (\ref{eq:relcans}), we may then write, for any $x\in X$
\begin{equation}\label{eq:relcans3}
\dim(\cM_x)=\pi^*K_q\cdot \tl{\Gamma}=\sum_{j \not \in I}(b_j-c_j)d_j+\sum_{i \in I}b_id_i,
\end{equation}
where the first equality follows from the fact that $\dim(\cM_x) =-K_X \cdot \Gamma -2= K_q \cdot f (s(\P^1))$ and the projection formula.

Since the minimality of $\tl{\Gamma}$ implies that $d_t\geq 0$ for all $t\in D$, we may conclude by showing the following:
\begin{itemize}
\item[(\ref{lem:RHVMRT1}.1)] $b_j-c_j>0$ for all $j \not \in I$.
\item[(\ref{lem:RHVMRT1}.2)] $\sum_{i \in I}b_id_i\geq\dim(\cM_x)$.
\end{itemize}

In order to prove (\ref{lem:RHVMRT1}.1), we notice that, having $b_j=c_j$ for some $j \not \in I$, is equivalent to say that the only positive roots having nonzero coefficient with respect to $\alpha_j$ are the elements of $\Phi^+(I)$. 

Since $G$ is the identity component of $\Aut(G/P)$, it follows that  every connected component of the Dynkin diagram $\cD$ of $G$ must contain a node corresponding to an index $i \in I$; in particular, the connected component $\cD'$ of $\cD$ containing the node corresponding to $\alpha_j$ contains a node corresponding to an index $i \in  I$. But then the longest root $\beta$ of the root subsystem  determined by $\cD'$ is a root of $\cD$, whose coordinates with respect to $\alpha_j$ and $\alpha_i$ are different from zero, hence $\beta\in\Phi^+\setminus\Phi^+(I)$, which contradicts that $b_j=c_j$.  

In order to show (\ref{lem:RHVMRT1}.2), we claim first that $d_i>0$ for every $i \in  I$. In fact, if  $d_i=0$ for some $ i\in I$, the $i$-th pivotal $\P^1$-bundle $\tl \cU_i$ which maps isomorphically into $f^*\cU$ via $\pi$, would be isomorphic to $\P^1\times\P^1$. But a section of $\tl \cU_i$ maps to $s(\P^1)$, which is contracted by $p \circ f$; hence the image of $\pi(\tl \cU_i)$ into $\cM$ would be a curve, whose points would parametrize the same curve in $X$, a contradiction.  

At this point, in order to conclude it is enough to check that 
$$\sum_{i \in I}b_i\geq \sharp(\Phi^+\setminus\Phi^+(I))=\dim(\cM_x).$$
The last equality follows from the hypothesis $\cM_x=G/P$. Since every root in $\Phi^+\setminus\Phi^+(I)$ increases at least in a unity at least an integer $b_i$, the first inequality follows, and the lemma is proved. 
\end{proof}

\begin{corollary}\label{cor:RHVMRT1}
Let $X$ be as in Theorem \ref{thm:RHVMRT}. Then, with the same notation as above, the morphism $p\circ\pi:\tl{\cU}\to\cM$ factors via a smooth $\P^1$-bundle $\tl{p}:\tl{\cU}\to \tl{\cM}$, for a smooth projective variety $\tl \cM$.
\end{corollary}

\begin{proof}
Write, as in Notation \ref{not:again}, any fiber $F'$ of $\pi:\tl \cU \to \cU$ as a quotient  $G'/B'$. By Lemma \ref{lem:RHVMRT1}, the fibers of the smooth morphism $p\circ\pi$ are isomorphic to the rational homogeneous manifold $G'/B'\times \P^1$. Moreover, since $\cM$ is rationally connected (because $\cU$ is rationally connected), hence simply connected, $p\circ\pi$ is defined by a cocycle $\theta'\in H^1(\cM,G'\times \Aut(\P^1))$. The image  of this cocycle via the natural map to $H^1(\cM,G')$ provides a $G'/B'$-bundle $\tl{\pi}:\tl{\cM}\to \cM$, whose pull-back to $\cU$ via $p$ is precisely $\pi$. 
\end{proof}

We may now finish the proof of the main result of this Section.

\begin{proof}[Proof of Theorem \ref{thm:RHVMRT}]
Note that, by construction, $\rho_{\tl{\cU}}=\rho_{G/B}+\rho_X=\rho_{G/B}+1$. On one hand, every minimal parabolic subgroup $P_i\supset B$, $i=1,\dots,k$ provides a smooth $\P^1$-bundle $\tl{\cU}\to \cU_i$. Moreover, the classes $\Gamma_i$ of the curves contracted by them are independent in $N_1(\tl{\cU})$. On the other hand, the class $\Gamma$ of a fiber of the morphism $\tl{p}$ provided by Corollary \ref{cor:RHVMRT1} does not lie in the hyperplane of $N_1(\tl{\cU})$ generated by the $\Gamma_i$'s, hence $\{\Gamma,\Gamma_1,\dots,\Gamma_k\}$ is a basis of $N_1(\tl{\cU})$. We may now apply Theorem \ref{thm:newOSWW}, which is a refined version of the Main Theorem in \cite{OSWW}, and claim that $\tl{\cU}$ is rational homogeneous. Then $X$, which is the image of a contraction of $\tl{\cU}$, is rational homogeneous, too. 
\end{proof}

\begin{remark}\label{rem:twofibrations}
Note that, along the proof of Theorem \ref{thm:RHVMRT} and its preliminary lemmata, we have not really used the assumption on $\cM$ being a component of $\rat^n(X)$. In fact we are only using that $p:\cU\to \cM$ is a $\P^1$-bundle with an evaluation morphism $q$ onto $X$, for which the natural map $\cM\to \rat^n(X)$ is surjective onto an irreducible component $\cM'$ and finite. In fact, it suffices to assume that $\cM\to \cM'$ is surjective and {\it generically} finite, since the existence of a curve contracted by $\cM\to \cM'$ easily implies the existence of a curve in a fiber of the $G/P$-bundle $q:\cU\to X$ contracted also by the natural morphism from $\cU$ to the universal family $\cU'$ over $\cM'$. Since this curve is free in $\cU$, it follows that $\dim(\cU)>\dim(\cU')$ and hence $\dim(\cM)>\dim(\cM')$, a contradiction. Hence, this observation allows us to state the following, slightly more general, result. 
\end{remark}

\begin{corollary}\label{cor:twofams}
Let $\cU$ be a smooth complex projective algebraic variety, and $p:\cU\to \cM$, $q:\cU\to X$ be two smooth fiber bundles, with fibers isomorphic to $\P^1$ and to a rational homogeneous manifold $G/P$, respectively. If moreover $b_2(X)=1$, and the natural map $\cM\to\rat^n(X)$ is a local isomorphism at a general point, then $X$, $\cM$ and $\cU$ are rational homogeneous.
\end{corollary}


\appendix

\section{Appendix: Characterization of 
flag manifolds revisited}\label{sec:OSWWrev}

The main result of \cite{OSWW} states that a Fano manifold whose elementary contractions are $\P^1$-bundles is necessarily rational homogeneous. In this Appendix we will prove a slightly more general form of that statement, namely:  

\begin{theorem}\label{thm:newOSWW}
Let $X$ be a smooth projective variety of Picard number $n$, satisfying that there exist $\Gamma_i\in N_1(X)$, $i=1,\dots,n$, independent $K_X$-negative classes generating $n$ extremal rays, whose associated elementary contractions $\pi_i:X\to X_i$ are smooth $\P^1$-bundles. Then $X$ is isomorphic to a flag manifold $G/B$, for some semisimple group $G$.
\end{theorem}

First of all we will show that, as in the proof of \cite[Theorem 1.2]{OSWW},
we may use the $\P^1$-bundles to construct a set of reflections in $N^1(X)$, generating a finite group that turns out to be the Weyl group of a semisimple Lie algebra $\fg$. We then call {\it homogeneous model of $X$}  the flag manifold $G/B$ associated to the corresponding Lie group $G$. One may then identify appropriately the relative anticanonical divisors $-K_i$ of the $\P^1$-bundles of $X$ with the relative anticanonical divisors $-\overline{K}_i$ of the homogeneous model $G/B$, via a linear isomorphism $\psi:N^1(X)\to N^1(G/B)$ preserving the cohomology of line bundles. 

We will then use this result to prove that the cone of curves of $X$ is generated by the classes of the $\Gamma_i$'s.
This implies that $X$ is a Fano manifold, whose elementary contractions are $\P^1$-bundles,
and the result follows then from \cite[Theorem 1.2]{OSWW}.

\subsection{The homogeneous model of $X$}

Following \cite{OSWW}, we consider $W$ to be the subgroup of $\GL(N^1(X))$ generated by the involutions of $N^1(X)$ defined by $r_i(L):=L+(L\cdot\Gamma_i)K_i$, where,  $K_i$ denotes the relative canonical bundle of $\pi_i$. Moreover, for every $i$ we consider the affine involution of $N^1(X)$ defined as the translation of $r_i$ by $K_X/2$, that is $r_i'(L):=r_i(L-K_X/2)+K_X/2$ (note that $K_i\cdot \Gamma_i=K_X\cdot\Gamma_i=-2$, for all $i$, so we may also write $r_i'(L):=L+(L\cdot\Gamma_i+1)K_i$). Let us denote by $W'$ the group generated by them, which is the translation by $K_X/2$ of the group $W$. Then the arguments provided in \cite[Section 2]{OSWW} work in our setting, and we may state the following:

\begin{lemma}\label{lem:newOSWW1}
Under the assumptions of Theorem \ref{thm:newOSWW},  
\begin{itemize}
\item[(a)]  $H^j(X,L)=H^{j+\sgn(L\cdot\Gamma_i+1)}(X,r_i'(L))$ for every $i=1,\dots, n$, $j\in \Z$.
\item[(b)] $\Chi(X,L)=\pm\Chi(X,w'(L))$, for every $L\in\Pic(X)$ and for every $w'\in W'$. 
\item[(c)] The group $W$ is finite. 
\end{itemize}
\end{lemma}

\begin{proof}
Follow verbatim Lemma 2.3, Lemma 2.7, Corollary 2.8 and Proposition 2.9 in \cite{OSWW}.
\end{proof}

\begin{corollary}\label{cor:newOSWW1}
With the same notation as above
\begin{itemize}
\item[(a)] $W$ is the Weyl group of a reduced root system $\Phi:=\{w(-K_i)|\,\,w\in W,\,\,i=1,\dots,n\}\subset N^1(X)$,
defining a semisimple Lie algebra $\fg$.
\item[(b)] The set $\Delta:=\{-K_i,\,\,i=1,\dots,n\}$ is a base of positive simple roots for $\Phi$, and the Cartan matrix of $\Phi$ with respect to $\Delta$ is the intersection matrix $(-K_i\cdot\Gamma_j)_{ij}$; in particular, this matrix is nonsingular. 
\item[(c)] Given the decomposition $\Phi=\Phi^+\cup \Phi^-$ determined by the base $\Delta$, the class $-K_X$ equals $\sum_{\beta\in\Phi^+}\beta$.
\end{itemize}
\end{corollary}

\begin{proof}
Items (a) and (b) follow by Corollary 2.10 and Proposition 2.13 in \cite{OSWW}. For the third part, note that  $B=\sum_{\beta\in\Phi^+}\beta$ 
satisfies $B\cdot \Gamma_i=2$, for all $i$ (see \cite[10.2, Cor. of Lemma B]{Hum2}). Since $-K_X$ satisfies the same property, the equality follows by (b).
\end{proof}

We may then consider a semisimple Lie group $G$ with Lie algebra $\fg$, and the corresponding flag manifold $\overline{X}=G/B$, that we call the {\it homogeneous model of }$X$. We denote its corresponding $\P^1$-bundles, fibers, relative canonicals, reflections, etc, by $\overline{\pi}_i:\overline{X}\to \overline{X}_i$, $\overline{\Gamma}_i$, $\overline{K}_i$, $\overline{r}_i$, etc.
By Corollary \ref{cor:newOSWW1}, we may now define linear isomorphisms: 
$$\psi:N^1(X)\to N^1(\overline{X}),\qquad \psi':N_1(X)\to N_1(\overline{X}),$$ 
sending $K_i$ to $\overline{K}_i$ and $\Gamma_i$ to $\overline{\Gamma}_i$, and, therefore, 
satisfying that $\psi(L)\cdot\psi'(\Gamma)=L\cdot\Gamma$, for all $L\in N^1(X) $,  $\Gamma\in N_1(X)$, and sending $\Phi$ to $\overline\Phi$. Then:
\begin{lemma}\label{lem:preserveK}
With the same notation as above, the isomorphism $\psi$ sends $K_X$ to $K_{\overline{X}}$. Moreover, we have isomorphisms of groups $W\cong \overline{W}$, $W'\cong \overline{W}\,\!'$, compatible via $\psi$ with their actions on $N^1(X)$ and $N^1(\overline{X})$. 
\end{lemma}
\begin{proof}
The first part follows from Corollary \ref{cor:newOSWW1} (c). For the second part, define the isomorphisms of groups by sending $r_i$ to $\overline{r}_i$ and $r'_i$ to $\overline{r}'_i$.
\end{proof}

One of our first goals will be to prove that the map $\psi$ preserves the cohomology of line bundles. In order to do that, we need first to prove that the dimension of $X$ equals the dimension of its homogeneous model.

\begin{remark}\label{rem:dimflags}
Note that the hypotheses in Theorem \ref{thm:newOSWW} suffice to claim that $X$ is rationally chain connected with respect to the curves $\Gamma_i$, $i=1,\dots,n$. In fact, we may consider the rationally connected fibration  with respect to the families of deformations of the $\Gamma_i$'s  (see \cite[IV.4.16]{kollar}), $\xymatrix{X  \ar@{-->}[r] & Z}$, which is a proper map from an open subset of $X$ onto its image. If $\dim(Z)>0$, then we may choose an effective divisor  on $Z$ and consider its strict transform $H$ in $X$. By construction, $H$ has intersection zero with every $\Gamma_i$, a contradiction. In particular we may claim that:
\begin{itemize}
\item $H^i(X,\cO_X)=0$ for $i> 0$ (cf \cite[IV.3.8]{kollar}), and that, by Serre duality,
\item $m:=\dim(X)$ is the only integer satisfying that $H^m(X,K_X)\neq 0$.
\end{itemize}
\end{remark}

\begin{corollary}\label{cor:dimflags}
With the same notation as above, $m=\dim(\overline{X})$, which is equal to the length of the longest element of $W$.
\end{corollary}

\begin{proof}
It is enough to find an element of the affine Weyl group $w'\in W'$ satisfying that $w'(\cO_X)=K_X$. In fact, if this is the case, the only index $m$ in which $K_X$ has nonzero cohomology would be determined by $w'$; since on the other hand we would also have $w'(\cO_{\overline{X}})=K_{\overline{X}}$ (by Lemma \ref{lem:preserveK}), the integer $m$ would be the same for $X$ and $\overline{X}$.

The condition required is equivalent to find an element $w\in W$ such that $w(-K_X)=K_X$.
Now, it is well known (see \cite[Section 1.8]{Hum3}) that the longest element of $W$ satisfies this property.   
\end{proof}

\begin{corollary}\label{cor:newOSWW2}
With the same notation as above, $\Chi(X,L)=\Chi(\overline{X},\psi(L))$, for all $L\in\Pic(X)$ satisfying that $\psi(L)\in\Pic(\overline{X})$.
\end{corollary}

\begin{proof}
As in \cite[Def. 2.6]{OSWW} we consider the polynomial function, of degree smaller than or equal to $m$, $\Chi_X:N^1(X) \to \R$, satisfying that $\Chi_X(L)=\Chi(X,L)$ for every $L\in\Pic(X)$. On one hand, one may prove that (cf. \cite[Prop. 2.9]{OSWW}) $\Chi_X$ vanishes on the hyperplanes $w'(K_X/2+(\Gamma_i)^\perp)$ for every $i$ and every $w'\in W'$, and there is one of these hyperplanes for every positive root $\beta\in \Phi^+$. Then $\Chi_X$ vanishes in at least $m=\dim(\overline{X})=\sharp(\Phi^+)$ hyperplanes of $N^1(X)$. Let $1+F_\beta=0$ denote the equation of the hyperplane determined by $\beta\in\Phi^+$; since the degree of $\Chi_X$ is not greater than $m$, it follows that 
$$\Chi_X=a\prod_{\beta\in\Phi^+}(1+F_\beta),$$
where $a=\Chi_X(\cO_X)$, which is equal to $1$ by Remark \ref{rem:dimflags}. In particular, $\Chi_X$ is completely determined by the root system $\Phi$, and our statement follows.  
\end{proof}

\begin{proposition}\label{prop:newOSWW3}
With the same notation as above, $h^i(X,L)$ is equal to $h^i(\overline{X},\psi(L))$, for all $L\in\Pic(X)\cap\psi^{-1}(\Pic(\overline{X}))$, and every integer $i$.
\end{proposition}   

\begin{proof}
Let us also denote by $N\subset N^1(X)$ the dual cone of the cone generated by $\Gamma_1,\dots,\Gamma_n$. Note that it is not true a priori --as it is in the case of a flag manifold $G/B$, and of any $X$ in the setup of \cite{OSWW}-- that $N=\Nef(X)$. On the other hand, arguing as in \cite[Cor. 2.19]{OSWW}, we may state that the interior of $N$ is a fundamental chamber for the action of $W$ on $N^1(X)$, and that the interior of $N':=N+K_X/2$ is a fundamental chamber for the action of $W'$ on $N^1(X)$. 

Let us consider a reduced sequence of reflections in $W'$, $(r'_{i_1},\dots, r'_{i_m})$, associated with simple roots $(-K_{i_1},\dots,\hspace{-1pt} -K_{i_m})$, whose composition is the longest element $w'_0$ of the group $W'$. 
Given a divisor $L\in N'\cap \Pic(X)$ satisfying that $\psi(L)\in \Pic(\overline{X})$, we may construct recursively the divisors $L_1:=r_{i_1}(L)$, $L_2:=r'_{i_2}\circ r'_{i_1}(L)$, $\dots$, $L_m:=w'_0(L)$. 
We consider also the corresponding sequence $(\overline{L}_1,\dots,\overline{L}_m)$ of divisors in $\overline{X}$ constructed upon $\overline{L}:=\psi(L)$ via the same sequence reflections; 
note that $\overline{L}_j=\psi(L_j)$ for all $j$, by Lemma \ref{lem:preserveK}. 
At every step, the cohomology of $L_{j+1}$ is obtained from the cohomology of $L_j$ by shifting its degree by $+1$ or $-1$, as described in Lemma \ref{lem:newOSWW1} (a), depending only on the intersection number $L_{j}\cdot\Gamma_{i_{j+1}}$. 
In the case of $\overline{X}=G/B$ the degree of the shifting is $+1$ at every step, by the Borel--Weil--Bott theorem. 
Since $L_{j}\cdot \Gamma_{i_{j+1}}=\psi(L_j)\cdot\psi'(\Gamma_{i_{j+1}})=\overline{L}_{j}\cdot \overline{\Gamma}_{i_{j+1}}$, it follows that at every step the degree of the shifting is $+1$ for $X$, as well. 
In particular, for any positive integer $i$, we have:
$$H^i(X,L)=H^{m+i}(X,w'_0(L))=0,$$
and we may conclude that:
$$
H^0(X,L)=\Chi(X,L)=\Chi(\overline{X},\psi(L))=H^0(\overline{X},\psi(L)).
$$
Finally, using that the interior of $N'$ is a fundamental chamber for the action of $W'$ on $N^1(X)$, we may use Lemma \ref{lem:newOSWW1} (a) to conclude that the statement holds for every divisor $L\in\Pic(X)\cap\psi^{-1}(\Pic(\overline{X}))$. 
\end{proof}

\subsection{The Mori cone of $X$}
We are now ready to finish the proof of Theorem \ref{thm:newOSWW}, by reducing it to \cite[Theorem 1.2]{OSWW} via the following description of the cone of effective curves of $X$: 

\begin{proposition}\label{prop:simpl} 
With the same notation as above, $\cNE{X}$ is the simplicial cone generated by $\Gamma_1,\dots,\Gamma_n$.
In particular $X$ is a Fano manifold.
\end{proposition}

To this end, we will use the Bott-Samelson varieties of $X$, whose construction we recall here (see \cite[Section 3]{OSWW}):

\begin{construction}
Given a sequence $\ell=(l_1,\dots,l_r)$, $l_i\in \{1,2,\dots,n\}$, set $\ell[s]:=(l_1,\dots,l_{r-s})$ for every $s\leq r$; in particular $\ell[0]=\ell$.
Fix a point $x \in X$. Given a sequence $\ell=(l_1,\dots,l_r)$ of indices in $\{1,2,\dots,n\}$, for every $s=0,\dots,r$ we construct a manifold $Z_{\ell[s]}$, $s=0,\dots,r$, called the {\it Bott-Samelson variety} associated to $\ell[s]$, together with morphisms
$$f_{\ell[s]}:Z_{\ell[s]} \to X,\quad p_{\ell[s+1]}:Z_{\ell[s]}\to Z_{\ell[s+1]},$$
which are defined recursively as follows: for $s=r$ we set $Z_{\ell[r]}:=\{x\}$ and let $f_{\ell[r]}:\{x\}\to X$ be the inclusion. Then for $s<r$ we consider the elementary contraction $\pi_{l_{r-s}}:X\to X_{l_{r-s}}$ determined by the extremal ray generated by $\Gamma_{l_{r-s}}$, the composition
$g_{\ell[s+1]}:=\pi_{l_{r-s}}\circ f_{\ell[s+1]}:Z_{\ell[s+1]}\to X_{l_{r-s}}$, and define $Z_{\ell[s]}$ upon $Z_{\ell[s+1]}$ as the fiber product with $\pi_{l_{r-s}}$:
$$
\xymatrix@=35pt{Z_{\ell[s]}\ar[r]^{f_{\ell[s]}}\ar[d]_{p_{\ell[s+1]}}&X\ar[d]^{\pi_{l_{r-s}}} \\
Z_{\ell[s+1]}
\ar[ru]_{f_{\ell[s+1]}}\ar[r]^{g_{\ell[s+1]}}&X_{l_{r-s}}}
$$
\end{construction} 

A key result for our proof of Proposition \ref{prop:simpl} will be the following:

\begin{proposition}\label{prop:stein2}[Cf. \cite[Corollary 3.18]{OSWW}]
Let $\ell=(l_1,\dots,l_r)$ be a sequence, with $l_j\in\{1,\dots,n\}$, for all $j$. Then $\dim f_\ell(Z_\ell) =\dim Z_\ell$ if and only if $w(\ell)$ is reduced.
\end{proposition}

Since, with our current assumptions, we do not have the ampleness of  $-K_X$, the proof given in \cite{OSWW}
needs to be slightly modified, by using weaker versions of \cite[Lemma 3.14 and Proposition 3.17]{OSWW}:

\begin{lemma}\label{lem:j1}
Let $D$ be a nef divisor on $Z_\ell$. Then $H^i(Z_\ell,f_\ell^*(K_{X}/2)+D)=0$ for every $i >0$. 
\end{lemma}

\begin{proof} See \cite[Proof of Lemma 3.14]{OSWW}.
\end{proof}

\begin{proposition}\label{prop:Chi}
Let $L\in\Pic(X)$. Then
$\Chi(Z_\ell, f^*_\ell L) =\Chi(\overline Z_\ell,   \overline f^*_\ell{\psi ( L)})$ and, if $L- K_{X}/2$ is nef, then
$h^0(Z_\ell,  f^*_\ell L) =h^0(\overline Z_\ell, \overline f^*_\ell{\psi ( L)})$. In particular, if $L$ is nef as well, then $f^*_\ell L$ is big if and only if  $\overline f^*_\ell{\psi ( L)}$ is big.
\end{proposition}

\begin{proof} The first part follows verbatim from the proof of \cite[Proposition 3.17]{OSWW}.
For the second part we notice that $\psi(\Nef(X))\subset \Nef(\overline{X})$, since this last cone is the dual of the cone generated by the $\overline{\Gamma}_i$'s, and that $\psi(K_X)=K_{\overline{X}}$. Therefore, if $L- K_{X}/2$ is nef, then the higher cohomologies of $f^*_{\ell} L$ and $\overline f^*_{\ell} \psi (L)$ vanish by Lemma \ref{lem:j1}. The last statement follows from \cite[Lemma~2.2.3]{L2}, by applying the argument above to $rL-K_X/2=(r-1)L+(L-K_X/2)$, for $r\gg 0$.
\end{proof}

\begin{proof}[Proof of Proposition \ref{prop:stein2}]
Let us consider the homogeneous model $\overline{X}$ of $X$ and the corresponding Bott-Samelson variety $\overline{Z}_\ell$, with evaluation $\overline{f}_\ell:\overline{Z}_\ell\to \overline{X}$. It is known that the property holds for $\o{proof}verline{f}_\ell$, since $ \overline f_\ell(\overline Z_\ell)$ is the Schubert variety $\overline{Bw(\ell)B}/B$ of $\overline{X}=G/B$.

Take an ample line bundle $L \in \Pic(X)$ satisfying that $L-K_X/2$ is nef. Since $\psi(\Nef(X))\subset \Nef(\overline{X})$ (see the proof of Proposition \ref{prop:Chi}), then $\psi(L)$ is ample on $\overline{X}$. Hence  $\overline{f}_\ell^*(\psi(L))$ is big if and only if $w(\ell)$ is reduced, by our first observation above. By Proposition \ref{prop:Chi}, we then have that $f^*_\ell L$ is big  if and only if $w(\ell)$ is reduced. Clearly $f^*_\ell L$ is big if and only if  $\dim f_\ell(Z_\ell) =\dim Z_\ell$, and the statement is proved.
\end{proof}

\begin{proof}[Proof of Proposition \ref{prop:simpl}]
In order to prove that $\cNE{X}$ is a simplicial cone generated by $\Gamma_1,\dots,\Gamma_n$, it is enough to show that for every proper subset $I\subset\{1, \dots, n\}$ there exists a contraction $\pi_I:X\to X_I$ satisfying that ${\pi_{I}}_*(\Gamma_i)=0$ if and only if $i\in I$.

Fix a proper subset $I \subset \{1, \dots, n\}$.  Following \cite[IV. Theorem 4.16]{kollar}, there exists a proper fibration $\pi_I:X^0 \to X_I^0$, defined on an open subset $X^0 \subset X$, whose fibers are $\ch(I)$-equivalence classes: by definition, two points in $X$ are in the same $\ch(I)$-equivalence class if and only if there exists a connected chain of rational curves, whose irreducible components are curves in the classes $\Gamma_i$, $i\in I$,
connecting them. 
In order to conclude the proof, it suffices to show that, $\pi_I$ is defined on $X$, that is $X^0=X$. But, by \cite[Proposition 1]{BCD}, the indeterminacy locus $X \setminus X^0$ of $\pi_I$ is contained in the union of $\ch(I)$-equivalence classes of dimension greater than the dimension of the general one, hence it is enough to prove that all the $\ch(I)$-equivalence classes have the same dimension.

Let $\ell=(l_1, \dots, l_r)$ be a reduced list such that $l_i \in I$, for all $i$, and $w(\ell)$ is the longest word in the subgroup $W_I\subset W$ generated by the reflections $r_i$, $i\in I$. Let $x \in X$ be any point, and $Z_\ell$ be the Bott-Samelson variety associated to $\ell$ such that $Z_{\ell[r]}=\{x\}$.

Clearly $f_\ell(Z_\ell)$ is contained in the $\ch(I)$-equivalence class containing $x$; let us show that
the opposite inclusion holds.
Assume, by contradiction, that this is not the case. Then there exists $l_{r+1}\in I$ such that, being $\ell'=(l_1,\dots,l_r,l_{r+1})$, we have $f_\ell(Z_\ell)\subsetneq f_{\ell'}(Z_{\ell'})$. Since these two varieties are irreducible, $\dim Z_\ell =\dim (f_\ell(Z_\ell))< \dim( f_{{\ell}'}(Z_{\ell'}))$, hence $\dim( f_{{\ell}'}(Z_{{\ell}'}))= \dim Z_{{\ell}'}$. This, by  Proposition  \ref{prop:stein2}, contradicts the fact that $\ell'$ is not reduced.

Then $f_\ell(Z_\ell)$ is  the $\ch(I)$-equivalence class containing $x$. In particular, since $\ell$ is reduced, we have $\dim f_\ell(Z_\ell) = \dim Z_\ell = r$, by Proposition  \ref{prop:stein2} again, and  all $\ch(I)$-equivalence class have the same dimension.
\end{proof}

\bibliographystyle{plain}
\bibliography{biblio}

\begin{thebibliography}{10}

\bibitem{Akh}
Dmitri~N. Akhiezer.
\newblock {\em Lie group actions in complex analysis}.
\newblock Aspects of Mathematics, E27. Friedr. Vieweg \& Sohn, Braunschweig,
  1995.

\bibitem{BPVV}
W.~Barth, C.~Peters, and A.~Van~de Ven.
\newblock {\em Compact complex surfaces}, volume~4 of {\em Ergebnisse der
  Mathematik und ihrer Grenzgebiete (3) [Results in Mathematics and Related
  Areas (3)]}.
\newblock Springer-Verlag, Berlin, 1984.

\bibitem{BCD}
Laurent Bonavero, Cinzia Casagrande, and St{\'e}phane Druel.
\newblock On covering and quasi-unsplit families of curves.
\newblock {\em J. Eur. Math. Soc. (JEMS)}, 9(1):45--57, 2007.

\bibitem{De}
Olivier Debarre.
\newblock {\em Higher-dimensional algebraic geometry}.
\newblock Universitext. Springer-Verlag, New York, 2001.

\bibitem{GHS}
Tom Graber, Joe Harris, and Jason Starr.
\newblock Families of rationally connected varieties.
\newblock {\em J. Amer. Math. Soc.}, 16(1):57--67 (electronic), 2003.

\bibitem{Gro1}
Alexander Grothendieck.
\newblock Sur la classification des fibr\'es holomorphes sur la sph\'ere de
  {R}iemann.
\newblock {\em Amer. J. Math.}, 79:121--138, 1957.

\bibitem{HH}
Jaehyun Hong and Jun-Muk Hwang.
\newblock Characterization of the rational homogeneous space associated to a
  long simple root by its variety of minimal rational tangents.
\newblock In {\em Algebraic geometry in {E}ast {A}sia---{H}anoi 2005},
  volume~50 of {\em Adv. Stud. Pure Math.}, pages 217--236. Math. Soc. Japan,
  Tokyo, 2008.

\bibitem{Huck}
Alan Huckleberry.
\newblock Actions of complex {L}ie groups and the {B}orel-{W}eil
  correspondence.
\newblock In {\em Symmetries in complex analysis}, volume 468 of {\em Contemp.
  Math.}, pages 99--123. Amer. Math. Soc., Providence, RI, 2008.

\bibitem{Hum1}
James~E. Humphreys.
\newblock {\em Linear algebraic groups}.
\newblock Springer-Verlag, New York-Heidelberg, 1975.
\newblock Graduate Texts in Mathematics, No. 21.

\bibitem{Hum2}
James~E. Humphreys.
\newblock {\em Introduction to {L}ie algebras and representation theory},
  volume~9 of {\em Graduate Texts in Mathematics}.
\newblock Springer-Verlag, New York, 1978.
\newblock Second printing, revised.

\bibitem{Hum3}
James~E. Humphreys.
\newblock {\em Reflection groups and {C}oxeter groups}, volume~29 of {\em
  Cambridge Studies in Advanced Mathematics}.
\newblock Cambridge University Press, Cambridge, 1990.

\bibitem{Hw4}
Jun-Muk Hwang.
\newblock Rigidity of rational homogeneous spaces.
\newblock In {\em International {C}ongress of {M}athematicians. {V}ol. {II}},
  pages 613--626. Eur. Math. Soc., Z\"urich, 2006.

\bibitem{kollar}
J{\'a}nos Koll{\'a}r.
\newblock {\em Rational curves on algebraic varieties}, volume~32 of {\em
  Ergebnisse der Mathematik und ihrer Grenzgebiete. 3. Folge. A Series of
  Modern Surveys in Mathematics [Results in Mathematics and Related Areas. 3rd
  Series. A Series of Modern Surveys in Mathematics]}.
\newblock Springer-Verlag, Berlin, 1996.

\bibitem{LM}
Joseph~M. Landsberg and Laurent Manivel.
\newblock On the projective geometry of rational homogeneous varieties.
\newblock {\em Comment. Math. Helv.}, 78(1):65--100, 2003.

\bibitem{L2}
Robert Lazarsfeld.
\newblock {\em Positivity in algebraic geometry. {I}}, volume~48 of {\em
  Ergebnisse der Mathematik und ihrer Grenzgebiete. 3. Folge. A Series of
  Modern Surveys in Mathematics [Results in Mathematics and Related Areas. 3rd
  Series. A Series of Modern Surveys in Mathematics]}.
\newblock Springer-Verlag, Berlin, 2004.
\newblock Classical setting: line bundles and linear series.

\bibitem{Mk3}
Ngaiming Mok.
\newblock Recognizing certain rational homogeneous manifolds of {P}icard number
  1 from their varieties of minimal rational tangents.
\newblock In {\em Third {I}nternational {C}ongress of {C}hinese
  {M}athematicians. {P}art 1}, in {\em AMS/IP Stud. Adv. Math., vol.42, pt.1}, pages 41--61. Amer. Math. Soc., Providence, RI, 2008.

\bibitem{MOSW}
Roberto Mu{\~n}oz, Gianluca Occhetta, Luis~E. Sol{\'a}~Conde, and Kiwamu
  Watanabe.
\newblock Rational curves, {D}ynkin diagrams and {F}ano manifolds with nef
  tangent bundle.
\newblock {\em Math. Ann.}, 361(3):583--609, 2015.

\bibitem{MOSWW}
Roberto Mu{\~n}oz, Gianluca Occhetta, Luis~E. Sol{\'a}~Conde, Kiwamu Watanabe,
  and Jaros{\l}aw~A. Wi\'sniewski.
\newblock A survey on the {C}ampana-{P}eternell conjecture.
\newblock In {\em Rend. Ist. Mat. Univ. Trieste}, volume~47, pages 127--185.
  EUT Edizioni Universit\`a di Trieste, 2015.

\bibitem{OSWW}
Gianluca Occhetta, Luis~E. Sol{\'a}~Conde, Kiwamu Watanabe, and Jaros{\l}aw~A.
  Wi{\'s}niewski.
\newblock {F}ano manifolds whose elementary contractions are smooth $\mathbb
  {P}^1$-fibrations: a geometric characterization of flag varieties.
\newblock Preprint arXiv:{\tt 1407.3658}. To appear in Ann. Sc. Norm. Super.
  Pisa Cl. Sci.

\bibitem{OSS}
Christian Okonek, Michael Schneider, and Heinz Spindler.
\newblock {\em Vector bundles on complex projective spaces}.
\newblock Progress in Mathematics, 3. Birkh\"auser, Boston, Mass., 1980.

\bibitem{GAGA}
Jean-Pierre Serre.
\newblock G\'eom\'etrie alg\'ebrique et g\'eom\'etrie analytique.
\newblock {\em Ann. Inst. Fourier, Grenoble}, 6:1--42, 1955--1956.

\end{thebibliography}

\end{document}